\documentclass[a4paper]{article}

\usepackage[utf8]{inputenc}

\usepackage{lmodern}
\usepackage{array}
\usepackage{graphicx}
\usepackage{stmaryrd}
\usepackage{subfig}
\usepackage{amssymb, amsfonts, amsthm, amsmath}
\usepackage{enumerate}
\usepackage{enumitem}
\usepackage{bbm}

\usepackage[english]{babel}

\usepackage{tikz}
\usetikzlibrary{decorations}

\usetikzlibrary{decorations.pathreplacing}

\tikzset{individu/.style={draw,thick}}

\numberwithin{equation}{section}

\theoremstyle{plain}
\newtheorem{theorem}{Theorem}[section]

\newtheorem{lemma}[theorem]{Lemma}
\newtheorem{proposition}[theorem]{Proposition}

\theoremstyle{definition}

\theoremstyle{remark}
\newtheorem{remark}[theorem]{Remark}

\newcommand{\N}{\mathbb{N}}
\renewcommand{\S}{\mathbb{S}}
\newcommand{\Z}{\mathbb{Z}}

\newcommand{\C}{\mathbb{C}}

\newcommand{\calF}{\mathcal{F}}

\newcommand{\calL}{\mathcal{L}}

\newcommand{\floor}[1]{{\left\lfloor #1 \right\rfloor}}

\renewcommand{\bar}[1]{\overline{#1}}
\renewcommand{\tilde}[1]{\widetilde{#1}}
\renewcommand{\epsilon}{\varepsilon}
\renewcommand{\phi}{\varphi}

\newcommand{\Addresses}{{
  \bigskip
  \footnotesize
  
  \textsc{Université Paris-Saclay, CNRS, CEA, Institut de physique théorique, 91191 Gif-sur-Yvette, France}\par\nopagebreak
  \textit{E-mail address}: \texttt{sanjay.ramassamy at ipht.fr}

}}

\title{Laminations of a graph on a pair of pants}
\author{Sanjay Ramassamy}
\date{\today}

\newcommand{\calP}{\mathcal{P}}

\begin{document}

\maketitle

\begin{abstract}
A lamination of a graph embedded on a surface is a collection of pairwise disjoint non-contractible simple closed curves drawn on the graph. In the case when the surface is a sphere with three punctures (a.k.a. a pair of pants), we first identify the lamination space of a graph embedded on that surface as a lattice polytope, then we characterize the polytopes that arise as the lamination space of some graph on a pair of pants. This characterizes the image of a simplified variant of the spectral map for the vector bundle Laplacian for a flat connection on a pair of pants. The proof uses a graph exploration technique akin to the peeling of planar maps.
\end{abstract}

\section{Introduction}

\subsection{Overview}

In this article we consider graphs that are embedded on a pair of pants, that is, the sphere $\S^2$ from which three points $P_1,P_2$ and $P_3$ have been removed. A lamination of such a graph is a collection of pairwise disjoint simple cycles on the graph such that each cycle encircles exactly one point  $P_i$. To a lamination we associate the triple of non-negative integers corresponding to the number of cycles in the lamination encircling each $P_i$.

The first result that we prove holds for a fixed graph, it characterizes all the possible triples of integers that one can obtain as one varies the lamination of the graph (Proposition~\ref{prop:onegraph}). This collection of triples is called the lamination space of the graph and we show that this lamination space is a convex lattice polytope characterized by six non-negative integers which arise as some geometric measurements from the graph.

Next we characterize all the lattice polytopes that can arise as the lamination space of some graph, by providing a necessary and sufficient set of inequalities that the aforementioned six integers associated to a graph should satisfy (Theorem~\ref{thm:charac}).

While the questions tackled and the answers provided pertain to the field of topological combinatorics, the motivation comes from the notion of spectral map associated to discrete operators on graphs, as explained in the following subsection. The discussion of spectral maps will be circumscribed to the next subsection and will not be mentioned again in the rest of the article.

\subsection{Motivation}

Consider a graph $G$ embedded on a pair of pants $\Sigma$, with edges carrying positive weights. We also equip $G$ with a flat $SU(2,\C)$ connection, that is, we attach an element of $SU(2,\C)$ to each directed edge, such that the two directions of an edge carry inverse elements of $SU(2,\C)$ and the product along any directed cycle on the graph which is contractible on $\Sigma$ is equal to the identity. The determinant of the vector-bundle Laplacian operator on $G$ can be expressed as a polynomial in three variables~\cite{Ke1}.

The map associating to $G$ this polynomial (or rather the algebraic surface obtained as the zero locus of this polynomial) is called the \emph{spectral map}. This map may be defined for operators other than the Laplacian and for surfaces other than the pair of pants~\cite{KO,GK,Foc,Ke1,Ke2,KW2,Ke4,George}. An important question is to characterize the image of the spectral map, namely the polynomials that can be realized by some graph. This question is very hard in general, and although it has been solved in some special cases (see the above references), it remains open in general.

One may consider a map which is simpler than the spectral map and which we call the \emph{topological spectral map}. We consider an unweighted graph rather than a weighted graph and we associate to it a lattice polytope rather than a polynomial. This lattice polytope is obtained by taking the Newton polytope of the polynomial associated to the graph carrying some generic edge weights. In the present case, the Newton polytope of a three-variable polynomial $P(X_1,X_2,X_3)$ is defined as the convex hull of all the triples $(m_1,m_2,m_3)$ such that the monomial $X_1^{m_1} X_2^{m_2} X_3^{m_3}$ has a non-zero coefficient in $P$. A simpler question than (and usually a prerequisite to) characterizing the image of the spectral map is to characterize the image of the topological spectral map. As explained in the next paragraph, in this article we solve this simpler question in the case of the vector-bundle Laplacian with a flat connection on a pair of pants and we hope than it may serve to solve the general question in future work.

Following pioneering work by Chaiken~\cite{Chaiken} and Forman~\cite{For} on the combinatorial study of the determinant of the vector-bundle Laplacian (which is itself a generalization of Kirchhoff's classical matrix-tree theorem~\cite{Ki}), Kenyon~\cite{Ke1} showed that the polynomial $P$ associated to a weighted graph $G$ on a pair of pants could be expressed as a sum over essential cycle-rooted spanning forests on $G$, that is, spanning forests where each connected component is a unicycle (a connected graph with as many vertices as edges) whose cycle is non-contractible on the pair of pants. As a consequence he deduced that the Newton polytope of the polynomial $P$ is equal to the lamination space of $G$ from which the point $(0,0,0)$ is removed. Since in the present article we characterize all the possible lamination spaces one may obtain as one varies the graph $G$, we are also characterizing the image of the topological spectral map for the vector-bundle Laplacian with a flat connection on a pair of pants.

\subsection*{Organization of the paper}

We introduce the relevant definitions and state our main results in Section~\ref{sec:mainresults}. In Section~\ref{sec:specialloops} we describe an exploration process of a graph on a pair of pants and use it to realize the lamination space of that graph as a polytope. In passing we define three collections of special loops and study their properties. In Section~\ref{sec:necessary} we derive some necessary conditions for the polytopes arising as the lamination space of some graph. We show in Section~\ref{sec:classofgraphs} that these conditions are sufficient by constructing a class of graphs having as a lamination space a given polytope satisfying the aforementioned conditions.

\section{Main results}
\label{sec:mainresults}

We consider the three-holed sphere $\Sigma$ obtained by removing from the sphere $\S^2$ three distinct points $P_1,P_2$ and $P_3$. Every simple closed curve $C$ on $\S^2$ which does not pass through the points $P_i$ separates $\S^2$ into two hemispheres. For every $1 \leq i \leq 3$, we denote by $H_i(C)$ (resp. $H'_i(C)$) the connected component of $\S^2 \setminus C$ which contains $P_i$ (resp. which does not contain $P_i$). A simple closed curve $C$ is called \emph{of type $i$} for $1 \leq i \leq 3$ if one of the hemispheres defined by $C$ contains $P_i$ and the other hemisphere contains the other two points, i.e. if
\[
H_i(C)=H'_{i+1}(C)=H'_{i+2}(C).
\]
In the previous equalities, as well as in the remainder of this article, the indices $1 \leq i \leq 3$ should be considered modulo $3$. We will also denote by $\bar{H_i}(C)$ and $\bar{H'_i}(C)$ the closed hemispheres (containing $C$ this time).

Let $G$ be a connected nonempty graph embedded in $\S^2$. The connected components of $\S^2 \setminus G$ are topological disks, they are called the faces of $G$ and we denote by $\calF$ the set of faces of $G$. We say that $G$ is a \emph{$\Sigma$-graph} if there exist three distinct faces $(F_1,F_2,F_3)\in\calF$ (called \emph{marked faces}) such that $P_i$ is in the interior of $F_i$ for all $1\leq i\leq 3$. A $\Sigma$-graph is more than just a graph embedded in $\Sigma$ because we require that the graph actually separates the three punctures. A \emph{lamination} of the $\Sigma$-graph $G$ is a collection $L$ of pairwise disjoint combinatorial simple cycles on $G$ such that each loop in $L$ is non-contractible on $\Sigma$. By disjoint we mean having no vertex in common. For any non-negative integers $m_1$, $m_2$ and $m_3$, a lamination is said to be \emph{of type $(m_1,m_2,m_3)$} if for any $1 \leq i \leq 3$ it contains $m_i$ loops of type $i$. The \emph{lamination space} $\calL(G)$ of a $\Sigma$-graph $G$ is defined to be the set of all $(m_1,m_2,m_3)\in (\Z_+)^3$ such that $G$ admits a lamination of type $(m_1,m_2,m_3)$. Below we will describe the lamination space of a given $\Sigma$-graph $G$ as the integer points of a lattice polytope defined in terms of some geometric characteristics of $G$.

We define a distance function $d_G$ on $\calF$ such that any two faces sharing a vertex are at distance $1$ for $d_G$. Let $G^*$ be the dual graph of $G$ (seen as a graph in $\S^2$). Construct $\tilde{G^*}$ by adding to $G^*$ an edge between any two vertices of $G^*$ such that the corresponding two faces of $G$ share a vertex in $G$. The distance $d_G$ is defined to be the usual graph distance on the vertex set of $\tilde{G^*}$, which is canonically in bijection with $\calF$. In the special case when all the vertices of $G$ have degree $3$ (in which case two faces share a vertex if and only if they share an edge), then $\tilde{G^*}=G^*$ and $d_G$ is the classical distance between two faces corresponding to the graph distance on the dual graph. From now on, whenever we mention the distance between two faces of $G$, the distance function will implicitly be $d_G$.

Define $d_1(G):=d_G(F_2,F_3)$, $d_2(G):=d_G(F_1,F_3)$ and $d_3(G):=d_G(F_1,F_2)$. Also, for any $1\leq i\leq 3$, define $M_i(G)$ to be the maximal number of pairwise disjoint simple loops of type $i$ one can simultaneously draw on $G$. Given a $\Sigma$-graph $G$, we define the sextuple
\[
\sigma(G):=(M_1(G),M_2(G),M_3(G),d_1(G),d_2(G),d_3(G))\in(\Z_+)^3 \times \N^3,
\]
where $\Z_+$ (resp. $\N$) denotes the set of all nonnegative (resp. positive) integers. See Figure~\ref{fig:facedistances} for an example.

\begin{figure}[htbp]
\centering
\includegraphics[height=1.3in]{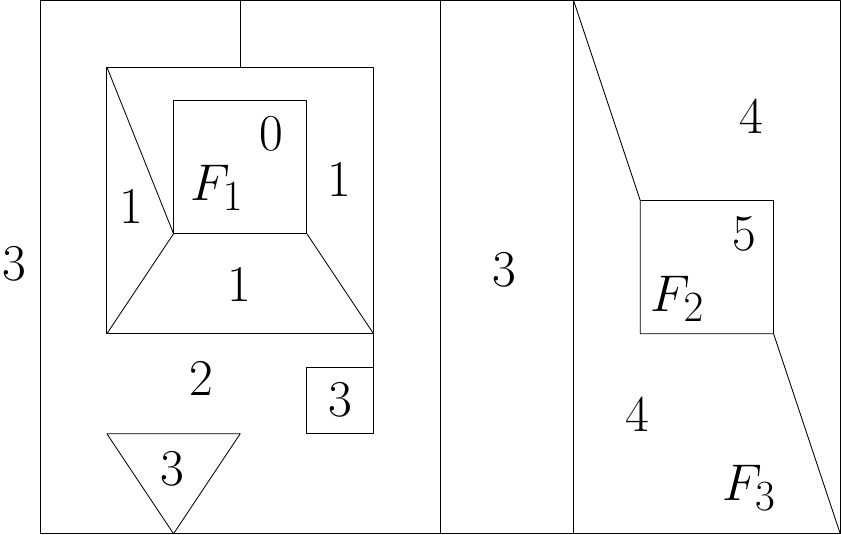}
\caption{A $\Sigma$-graph $G$ with each face labelled by its distance to the marked face $F_1$. For this graph, $\sigma(G)=(4,1,1,1,4,5)$.}
\label{fig:facedistances}
\end{figure}

Given a sextuple of integers $\tau=(a,b,c,d,e,f)\in(\Z_+)^3 \times \N^3$, we define the convex lattice polytope $\calP_\tau$ by
\[
\calP_\tau:=\left\{(x,y,z)\in(\Z_+)^3\vert x\leq a,\ y\leq b,\ z\leq c,\ y+z\leq d,\ x+z\leq e,\ x+y\leq f\right\}.
\]

\begin{proposition}
\label{prop:onegraph}
For any $\Sigma$-graph $G$, its lamination space $\calL(G)$ is the polytope $\calP_{\sigma(G)}$.
\end{proposition}

Proposition~\ref{prop:onegraph} is proved in Section~\ref{sec:specialloops}.

\begin{remark}
The inequalities $m_i\leq M_i(G)$ are not redundant with the inequalities $m_i+m_{i+1}\leq d_{i+2}(G)$, as illustrated by Figure~\ref{fig:nonredundant}. On that picture, $d_1(G)=d_2(G)=d_3(G)=2$ and $M_1(G)=M_2(G)=M_3(G)=1$. The triple $(m_1,m_2,m_3)=(2,0,0)$ verifies the inequalities $m_i+m_{i+1}\leq d_{i+2}$, but that graph has no lamination of type $(2,0,0)$. This proposition corrects a statement made in~\cite{Ke1}, which claimed that the inequalities $m_i+m_{i+1}\leq d_{i+2}(G)$ alone define the lamination space, without mention of the inequalities $m_i\leq M_i(G)$.
\end{remark}

\begin{figure}[htbp]
\centering
\includegraphics[height=.8in]{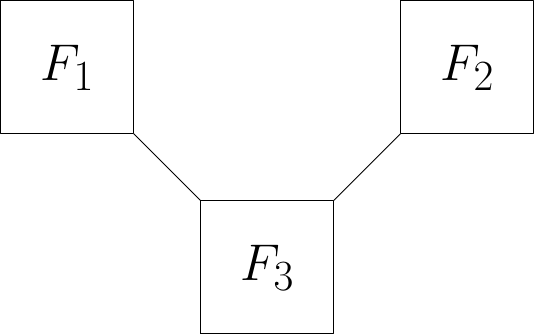}
\caption{An example of a graph $G$ illustrating the need to require the inequalities $m_i \leq M_i(G)$ in order to characterize the types of laminations that can arise.}
\label{fig:nonredundant}
\end{figure}

We can now characterize all the convex lattice polytopes that arise as the lamination space of some $\Sigma$-graph. By the previous proposition, it suffices to characterize the sextuples $\tau$ that arise as some $\sigma(G)$.

\begin{theorem}
\label{thm:charac}
Fix $\tau=(\mu_1,\mu_2,\mu_3,\delta_1,\delta_2,\delta_3)\in(\Z_+)^3 \times \N^3$. There exists a $\Sigma$-graph $G$ such that $\sigma(G)=\tau$ if and only if the following inequalities hold for all $1\leq i\leq 3$:
\begin{enumerate}[label=($T_{\arabic*}$)]
\item $\max(\mu_{i+1},\mu_{i+2})\leq \delta_i \leq \mu_{i+1}+\mu_{i+2}$;
\item $\delta_{i+1}+\delta_{i+2}\leq 2 \mu_i + \delta_i +1$.
\end{enumerate}
\end{theorem}

The fact that conditions $(T_1)$ and $(T_2)$ are necessary is proved in Section~\ref{sec:necessary}, while the fact that they are sufficient is proved in Section~\ref{sec:classofgraphs} by explicitly constructing a $\Sigma$-graph $G$ such that $\sigma(G)=\tau$ whenever $\tau$ satisfies the two conditions.

A notion that will be important in the proof of our results is that of \emph{special loops}, which we informally define here (see Section~\ref{sec:specialloops} for the precise definition). Let $G$ be a $\sigma$-graph and let $1\leq i \leq 3$. The special loop $C_i^1$ is the tightest simple loop of type $i$ that can be drawn on $G$, tightest in the sense that it is the closest to $P_i$. If $2\leq k\leq M_i(G)$, the special loop $C_i^k$ is the tightest simple loop of type $i$ that can be drawn on $G$  while having no vertex in common with $C_i^1,\ldots,C_i^{k-1}$. See Figure~\ref{fig:specialloops} for an illustration. We define
\[
n_i(G):=M_{i+1}(G)+M_{i+2}(G)-d_i(G).
\]
We will see in Section~\ref{sec:specialloops} that the non-negative integer $n_i(G)$, called the \emph{depth of intersection} of the special loops of types $i+1$ and $i+2$ can be interpreted geometrically as the number of special loops of type $i+1$ (resp. $i+2$) having at least one vertex in common with some special loop of type $i+2$ (resp.  $i+1$). Defining the sextuple
\[
\sigma'(G):=(M_1(G),M_2(G),M_3(G),n_1(G),n_2(G),n_3(G))\in(\Z_+)^6,
\]
we can rephrase Theorem~\ref{thm:charac} with the variables $n_i$ replacing the variables $d_i$, producing more compact inequalities.

\begin{figure}[htbp]
\centering
\includegraphics[height=1.5in]{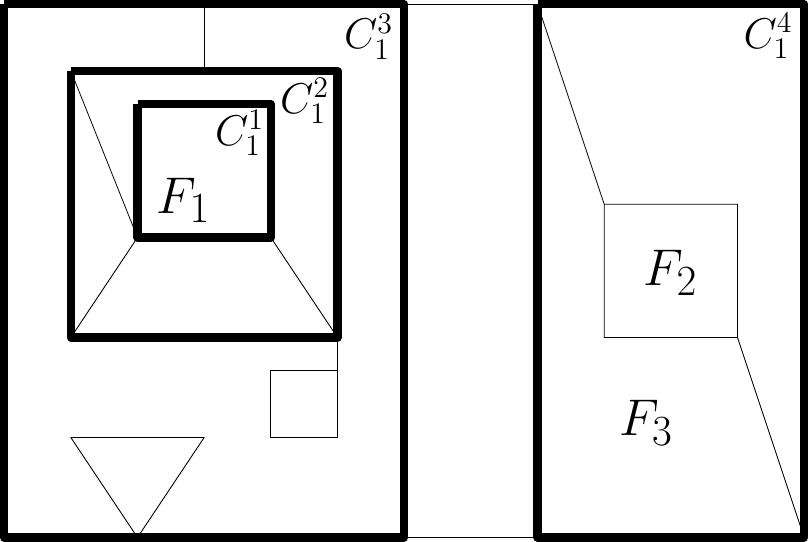}
\caption{Representation in bold of the special loops $C_1^k$ for the graph $G$ of Figure~\ref{fig:facedistances}.}
\label{fig:specialloops}
\end{figure}

\begin{theorem}
\label{thm:characbis}
Fix $\tau'=(\mu_1,\mu_2,\mu_3,\nu_1,\nu_2,\nu_3)\in(\Z_+)^6$. There exists a $\Sigma$-graph $G$ such that $\sigma'(G)=\tau'$ if and only if for all $1\leq i\leq 3$ we have
\begin{equation}
\label{eq:nureformulation}
0 \leq \nu_i \leq \min(\mu_{i+1},\mu_{i+2},\nu_{i+1}+\nu_{i+2}+1).
\end{equation}
\end{theorem}

It is not hard to see that Theorem~\ref{thm:characbis} is an immediate consequence of Theorem~\ref{thm:charac}.

In order to construct the special loops, we will explore any $\Sigma$-graph $G$ starting from the face $F_1$, discover a first layer consisting of the faces at distance $1$ from $F_1$, then a second layer consisting of the faces at distance $2$, etc. We will perform the same exploration starting from the faces $F_2$ and $F_3$ and understand how the boundaries of the layers arising in each of these three explorations interact with each other. In the case of simple triangulations, our exploration process coincides with the layer decomposition developed by Krikun~\cite{Kr} for infinite triangulations. More generally, this construction resembles the peeling process for planar maps (see for example~\cite{C}). The difference is that here we use a distance which differs slightly from the graph distance on the dual graph. Instead of peeling an edge by discovering the face on the other side of the edge, we are peeling a vertex, by discovering all the unknown faces containing a vertex which is on the boundary of what we have already explored. In the case of graphs with all the vertices having degree $3$, our exploration process coincides with the edge peeling process.

\section{Special loops around a puncture}
\label{sec:specialloops}

In this section we first describe an exploration process of a $\Sigma$-graph $G$ starting from a marked face, which will trace out a collection of special loops on $G$ centered around a marked face. The key result of this section, which will be used repeatedly in the rest of the article, is Lemma~\ref{lem:simpleloopstructure} describing the properties of the boundaries encountered during the exploration process. Then we will study how two collections of special loops intersect each other and deduce from this a proof of Proposition~\ref{prop:onegraph}.

\subsection{A collection of special loops around a puncture}
\label{subsec:specialloops}

We start by an elementary observation, which we will be using several times. Let $G$ be a connected planar graph and $\tilde{G}$ be a subgraph of $G$ (not necessarily an induced subgraph of $G$). One defines the distance function $d_{\tilde{G}}$ on the set of connected components of $\S^2 \setminus \tilde{G}$ in exactly the same way as the distance $d_G$ was defined on the faces of $G$. Note that the connected components of $\S^2 \setminus \tilde{G}$ do not have to be topological disks, they may be disks with multiple punctures or even the whole sphere if $\tilde{G}$ is empty. Then we have the following result.

\begin{lemma}
\label{lem:distances}
Let $G$ be a connected planar graph and $\tilde{G}$ be a subgraph of $G$. Let $F$ and $F'$ be two faces of $G$ and let $\tilde{F}$ and $\tilde{F'}$ be the two connected components of $\S^2 \setminus \tilde{G}$ containing respectively $F$ and $F'$. Then $d_{\tilde{G}}(\tilde{F},\tilde{F'}) \leq d_G(F,F')$.
\end{lemma}

\begin{proof}
Observe that $d(F,F')$ is the smallest value of $N$ for which there exists a sequence of faces $(F_0,\ldots,F_N)$ such that $F_k$ and $F_{k+1}$ have a common vertex for all $0\leq k\leq N-1$, $F_0=F$ and $F_N=F'$. A similar characterization holds for $d_{\tilde{G}}$, replacing faces by connected components of $\S^2 \setminus \tilde{G}$. Set $N=d(F,F')$ and pick a sequence $(F_0,\ldots,F_N)$ as above. Considering the sequence $(\tilde{F}_0,\ldots,\tilde{F}_N)$, where $\tilde{F}_k$ denotes the connected component of $\S^2 \setminus \tilde{G}$ containing $F_k$, we conclude that $d_{\tilde{G}}(\tilde{F},\tilde{F'})\leq N$.
\end{proof}

Let $G$ be a $\Sigma$-graph. For any $k\geq0$ and $1\leq i\leq 3$, define
\begin{equation}
A^k_i=\left\{F\in\calF\vert d_G(F,F_i)=k\right\}.
\end{equation}
For any $k\geq1$ and $1\leq i\leq 3$ such that $A^k_i$ is nonempty, define $B^k_i$ to be the boundary of the set $\bigcup_{j=0}^{k-1} A^j_i$ of faces that are at distance less than $k$ to $F_i$. Each $B^k_i$ is the union of simple loops that are pairwise edge-disjoint but not necessarily pairwise vertex-disjoint. The case when $B^k_i$ consists in the union of several loops corresponds to a branching event in the peeling terminology, see e.g.~\cite{BCK}. The following lemma describes structural properties of the simple loops in $B^k_i$ and will be used repeatedly in the remainder of the article.

\begin{lemma}
\label{lem:simpleloopstructure}
Let $G$ be a $\Sigma$-graph, let $1 \leq i\leq 3$ and let $k \geq 1$ be such that $B^k_i$ is defined. Then, we have
\begin{equation}
\label{eq:universes}
\bigcup_{j \geq k} A^j_i = \bigcup_{\substack{C \subset B^k_i \\ C \text{ simple loop}}} \bar{H'_i}(C).
\end{equation}
Furthermore, if $C$ and $C'$ are two distinct simple loops contained in $B^k_i$, then $H'_i(C) \cap H'_i(C') = \emptyset$.
Finally, if $C\subset B^k_i$ is a simple loop, then the faces in $\bar{H_i}(C)$ sharing an edge with $C$ are in $A^{k-1}_i$.
\end{lemma}

\begin{proof}
Let $C \subset B^k_i$ be a simple loop and assume there is a face $F\subset \bar{H'_i}(C)$ such that $d_G(F_i,F) \leq k-1$. Then one can find a sequence $(F^0,\ldots,F^n)$ of faces such that $n=d_G(F_i,F)$, $F^0=F_i$, $F^n=F$ and for every $1 \leq j \leq n$, the faces $F^{j-1}$ and $F^j$ share a vertex. By construction we have that for every $1 \leq j  \leq n$, $d_G(F_i,F^j)=j$. Denoting by $k'$ the largest $j$ such that $F^j\subset \bar{H_i}(C)$, we have that $d_G(F_i,F^{k'}) \leq k-2$ and by connectedness of the path of faces from the hemisphere $H_i(C)$ to the hemisphere $H'_i(C)$, $F^{k'}$ shares a vertex $v$ with $C$. This yields a contradiction because all the faces containing the vertex $v$ are at distance at most $k-1$ of $F_i$, hence $B^k_i$ cannot pass through $v$ so $v$ cannot lie on $C$. We deduce from this that
\[
\bigcup_{\substack{C \subset B^k_i \\ C \text{ simple loop}}} \bar{H'_i}(C) \subset \bigcup_{j \geq k} A^j_i.
\]
Furthermore, by construction, every edge in $B^k_i$ has on one side a face in $A^{k-1}_i$ and on the other side a face in $A^k_i$. This implies that all the faces in $\bar{H_i}(C)$ that contain an edge in $C$ must be in $A^{k-1}_i$.

If $C$ and $C'$ are two distinct simple loops contained in $B^k_i$ and $H'_i(C) \cap H'_i(C') \neq\emptyset$ then we can find a face $F$ which satisfies one of the following two conditions:
\begin{enumerate}
 \item $F\subset \bar{H_i}(C) \cap \bar{H'_i}(C')$ and $F$ shares an edge with $C$ ;
 \item $F\subset \bar{H_i}(C') \cap \bar{H'_i}(C)$ and $F$ shares an edge with $C'$. 
\end{enumerate}
This yields a contradiction because it implies on the one hand that $d_G(F,F_i) = k-1$ and on the other hand that $d_G(F,F_i) \geq k$. Thus $H'_i(C)$ and $H'_i(C')$ must be disjoint.

Finally, let $F$ be a face in $A^j_i$ with $j\geq k$. We construct $F^c$ the connected component of $ \bigcup_{j \geq k} A^j_i$ containing $F$ as follows. We say that two faces in $\bigcup_{j \geq k} A^j_i$ are neighbors if they share an edge (not just a vertex) and $F^c$ is the set of all faces in $\bigcup_{j \geq k} A^j_i$ that are reachable from $F$ by walking across neighboring faces (these intermediate faces on the path must also lie in $\bigcup_{j \geq k} A^j_i$). Then $F^c$ is a connected set, with boundary denoted by $B^c$. There exists a simple loop $C\subset B^c$ such that $F^c \subset \bar{H'_i}(C)$. By construction of $B^c$, for every edge $e$ of $C$, the face adjacent to $e$ in the hemisphere $\bar{H_i}(C)$ is in $A^{k-1}_i$ and the face adjacent to $e$ in the hemisphere $\bar{H'_i}(C)$ is in $A^k_i$. Hence $C\subset B^k_i$. We conclude that
\[
\bigcup_{j \geq k} A^j_i \subset \bigcup_{\substack{C \subset B^k_i \\ C \text{ simple loop}}} \bar{H'_i}(C).
\]
\end{proof}

Fix $1 \leq  i \leq 3$ and recall that the indices $i+1$ and $i+2$ are considered modulo $3$. Since $d_G(F_i,F_{i+1})=d_{i+2}(G)$, Lemma~\ref{lem:simpleloopstructure} implies for any $1\leq k \leq d_{i+2}(G)$ the existence of a unique simple loop $C_{i,i+1}^k \subset B_i^k$ such that $F_{i+1} \subset \bar{H'_i}(C_{i,i+1}^k)$. The uniqueness follows from the fact that if $C$ and $C'$ are two distinct simple loops contained in $B^k_i$, then $H'_i(C) \cap H'_i(C') = \emptyset$. Hence at each step $1\leq k \leq d_{i+2}(G)$ there cannot be more than one loop $C$ in $B^k_i$ such that $F_{i+1} \subset \bar{H'_i}(C)$. Similarly, for any $1\leq k \leq d_{i+1}(G)$ there exists a unique simple loop $C_{i,i+2}^k \subset B_i^k$ such that $F_{i+2} \subset \bar{H'_i}(C_{i,i+2}^k)$. For $1\leq k \leq \min(d_{i+1}(G),d_{i+2}(G))-1$, if $C_{i,i+1}^k \neq C_{i,i+2}^k$, then $C_{i,i+1}^{k+1} \neq C_{i,i+2}^{k+1}$ by Lemma~\ref{lem:simpleloopstructure}. Thus there exists a unique integer $\tilde{M_i}\geq0$ such that for any $1 \leq k \leq \tilde{M_i}$ we have $C_{i,i+1}^k = C_{i,i+2}^k$ and for any $\tilde{M_i}+1 \leq k \leq \min(d_{i+1}(G),d_{i+2}(G))$ we have $C_{i,i+1}^k \neq C_{i,i+2}^k$. If $1 \leq k \leq \tilde{M_i}$ we denote simply by $C_i^k$ the simple loop $C_{i,i+1}^k = C_{i,i+2}^k$. The following lemma gives the value of $\tilde{M_i}$.

\begin{lemma}
\label{lem:specialloopdef}
Let $G$ be a $\Sigma$-graph. For any $1 \leq i\leq 3$, we have $\tilde{M_i}=M_i(G)$. Furthermore, for any fixed $1\leq i\leq 3$, the loops $(C^k_i)_{1\leq k \leq M_i(G)}$ are pairwise disjoint.
\end{lemma}

\begin{proof}
Fix $1 \leq i \leq 3$. Consider $C_i^k$ and $C_i^{k'}$ for some $1 \leq k < k' \leq \tilde{M_i}$. First observe that these loops are nested, i.e. $H_i(C_i^k) \subset H_i(C_i^{k'})$. If there exists a vertex $v$ lying on both these loops, denoting by $e$ an edge of $C_i^k$ containing $v$ and by $F$ the face containing $e$ and lying in $\bar{H_i}(C_i^k)$, it follows from Lemma~\ref{lem:simpleloopstructure} that $d_G(F,F_i)=k-1$, which contradicts the fact that $F$ shares the vertex $v$ with some face $F'$ lying in $\bar{H'_i}(C_i^{k'})$ because $d_G(F',F_i)\geq k' > k$ again by Lemma~\ref{lem:simpleloopstructure}. So the loops $(C^k_i)_{1\leq k \leq \tilde{M_i}}$ are pairwise disjoint.

It remains to prove that $\tilde{M_i} = M_i(G)$. Since the simple loops $C^k_i$ of type $i$ are pairwise disjoint for $1 \leq k\leq \tilde{M_i}$, their union constitutes a lamination with $\tilde{M_i}$ loops of type $i$ hence $\tilde{M_i}\leq M_i(G)$.

Let $L$ be a lamination consisting in $M_i(G)$ simple loops of type $i$ denoted by $C^k$, $1\leq k \leq M_i(G)$, which are nested in such a way that for any $1 \leq k \leq M_i(G)-1$, $H_i(C^k) \subset H_i(C^{k+1})$. If $F$ is a face in $\bar{H'_i}(C^k)$ for some $1 \leq k \leq M_i(G)$, then by Lemma~\ref{lem:distances}, we have that  $d_G(F_i,F) \geq k$ hence
\begin{equation}
\label{eq:outsidealoop}
\bar{H'_i}(C^k) \subset \bigcup_{j \geq k} A^j_i.
\end{equation}
This implies that
\[
\bigcup_{j=0}^{k-1} A^j_i \subset \bar{H_i}(C^k).
\]
Recalling that $B^k_i$ is defined as the boundary of $\bigcup_{j=0}^{k-1} A^j_i$, we deduce that $B^k_i$ is well-defined for all $1 \leq k \leq M_i(G)$ and that $B^k_i \subset  \bar{H_i}(C^k)$. Since $C^k$ is of type $i$, this implies that any simple loop contained in $B^k_i$ is either contractible or of type $i$. So $k \leq \tilde{M_i}$. This statement holds for every $1\leq k\leq M_i(G)$ so $M_i(G) \leq \tilde{M_i}$.
\end{proof}

\begin{remark}
\label{rem:simplelooptypes}
For any $1 \leq k \leq M_i(G)$, the simple loop $C_i^k$ is of type $i$, for any $M_i(G)+1 \leq k \leq d_{i+1}(G)$ the simple loop $C_{i,i+1}^k$ is of type $i+1$ and for any $M_i(G)+1 \leq k \leq d_{i+2}(G)$ the simple loop $C_{i,i+2}^k$ is of type $i+2$.
\end{remark}

The loops $C^k_i$ are called \emph{special loops of type $i$}. These special loops are optimal if one wants to pack the maximum possible number of disjoint simple loops of a given type. For example $C^1_i$ is the ``tightest'' simple loop of type $i$ one can draw, $C^2_i$ is the ``tightest'' simple loop of type $i$ one can draw which would be disjoint from $C^1_i$, etc. See Figure~\ref{fig:specialloops} for an illustration.

\subsection{Intersection of two collections of special loops}

We will now describe how two collections of special loops of two different types intersect each other.

\begin{lemma}
\label{lem:interaction}
Let $G$ be a $\Sigma$-graph and let $1 \leq i \leq 3$. Fix two integers $1\leq j \leq d_i(G)$ and $1\leq k \leq d_i(G)$. Then $C_{i+1}^k \cap C_{i+2}^j = \emptyset$ if and only if $j+k \leq d_i(G)$. Furthermore, $H_{i+1}(C_{i+1}^k) \cap H_{i+2}(C_{i+2}^{d_i(G)+1-k}) = \emptyset$.
\end{lemma}

\begin{proof}

Assume that $C_{i+1}^k \cap C_{i+2}^j \neq \emptyset$. Let $v$ be a vertex in $C_{i+1}^k \cap C_{i+2}^j$, $e$ be an edge in $C_{i+1}^k$ containing $v$, $e'$ be an edge in $C_{i+2}^j$, $F$ be the face in $\bar{H_{i+1}}(C_{i+1}^k)$ containing $e$ and $F'$ be the face in $\bar{H_{i+2}}(C_{i+2}^j)$ containing $e'$. Then by Lemma~\ref{lem:simpleloopstructure}, $d_G(F_{i+1},F)=k-1$ and $d_G(F_{i+2},F')=j-1$. Since $F$ and $F'$ share the vertex $v$, we also have $d_G(F,F')=1$. By the triangle inequality, we conclude that $d_G(F_{i+1},F_{i+2}) \leq j+k-1$. Thus $j+k > d_i(G)$.

Conversely, assume that $C_{i+1}^k \cap C_{i+2}^j = \emptyset$. Then either $C_{i+2}^j \subset H_{i+2}(C_{i+1}^k)$ or $C_{i+2}^j \subset H_{i+1}(C_{i+1}^k)$. The latter alternative cannot be true, otherwise we would have $F_i \subset H_i(C_{i+1}^k) = H_{i+2}(C_{i+1}^k) \subset H_{i+2}(C_{i+2}^j)$, which would entail that $C_{i+2}^j$ is either contractible or of type $i+1$. Hence $C_{i+2}^j \subset H_{i+2}(C_{i+1}^k)$. Furthermore, as observed in the proof of Lemma~\ref{lem:specialloopdef}, the special loops of a given type are nested and disjoint, which implies that the loops $C_{i+1}^1,\ldots,C_{i+1}^k,C_{i+2}^1,\ldots,C_{i+2}^j)$ are pairwise disjoint, thus they form a lamination $L$. Let $F_{i+1}'$ (resp. $F_{i+2}'$) denote the connected component of $\S^2 \setminus L$ containing $P_{i+1}$ (resp. $P_{i+2}$) Then by Lemma~\ref{lem:distances}, since $L$ is a subgraph of $G$, we have $d_i(G)=d_G(F_{i+1},F_{i+2}) \geq d_L(F'_{i+1},F'_{i+2})=j+k$.

Finally, assume that $H_{i+1}(C_{i+1}^k) \cap H_{i+2}(C_{i+2}^{d_i(G)+1-k}) \neq \emptyset$. Then we can find a face $F\subset \bar{H_{i+1}}(C_{i+1}^k) \cap \bar{H_{i+2}}(C_{i+2}^{d_i(G)+1-k})$ which has at least one edge in common with $C_{i+2}^j$. By Lemma~\ref{lem:simpleloopstructure}, we have $F\in A_{i+2}^{d_i(G)-k}$ so $B_{i+2}^{d_i(G)-k}$ intersects $F$ (this intersection may be just a single vertex). On the other hand, it follows from Lemma~\ref{lem:simpleloopstructure} that $A_{i+2}^{d_i(G)-k-1} \subset H_{i+2}(C_{i+2}^{d_i(G)-k})$ so $B_{i+2}^{d_i(G)-k}\subset \bar{H_{i+2}}(C_{i+2}^{d_i(G)-k})$. Thus we obtain that $\bar{H_{i+1}}(C_{i+1}^k) \cap \bar{H_{i+2}}(C_{i+2}^{d_i(G)-k}) \neq \emptyset$, and the previous paragraph would entail that $k + (d_i(G)-k) > d_i(G)$. This is the desired contradiction, hence $H_{i+1}(C_{i+1}^k) \cap H_{i+2}(C_{i+2}^{d_i(G)+1-k}) = \emptyset$.
\end{proof}

It follows from Lemma~\ref{lem:interaction} that the nonnegative integer
\[
n_i(G)=M_{i+1}(G)+M_{i+2}(G)-d_i(G)
\]
counts the number of special loops of type $i+1$ (resp. $i+2$) which intersect some special loop of type $i+2$ (resp. $i+1$). For every $1 \leq i \leq 3$, we will call this integer $n_i(G)$ the \emph{depth of intersection} of the special loops of types $i+1$ and $i+2$.

We use the properties of these special loops to prove Proposition~\ref{prop:onegraph}.

\begin{proof}[Proof of Proposition~\ref{prop:onegraph}]
Assume $G$ has a lamination $L$ of type $(m_1,m_2,m_3)$. Let $1 \leq i \leq 3$. Then $m_i \leq M_i(G)$ by definition of $M_i(G)$. Furthermore, as in the proof of Lemma~\ref{lem:interaction}, we have $m_{i+1}+m_{i+2} \leq d_i(G)$. Thus $(m_1,m_2,m_3)\in \calP_{\sigma(G)}$

Conversely, assume that we have a triple of integers $(m_1,m_2,m_3)$ satisfying the six inequalities defining $\calP_{\sigma(G)}$. Set
\[
L=\left\{C^1_1,\ldots,C^{m_1}_1,C^1_2,\ldots,C^{m_2}_2,C^1_3,\ldots,C^{m_3}_3\right\}.
\]
Observe that for every $1\leq i \leq 3$, $C_i^k$ is well-defined because $k\leq m_i\leq M_i(G)$. By Lemma~\ref{lem:interaction}, the fact that $m_{i+1}+m_{i+2} \leq d_i(G)$ for every $i$ implies that these loops are pairwise disjoint. So $L$ is a lamination and its type is $(m_1,m_2,m_3)$ by construction.
\end{proof}

\section{Necessity of conditions $(T_1)$ and $(T_2)$}
\label{sec:necessary}

In this section, we prove one direction of Theorem~\ref{thm:charac}. Let $G$ be a $\Sigma$-graph. In order to alleviate notation, we will drop the dependency of $M_i$, $d_i$ and $n_i$ on $G$ in this section. We will show that the six components of $\sigma(G)$ satisfy the inequalities $(T_1)$ and $(T_2)$ of Theorem~\ref{thm:charac}. By symmetry it suffices to consider the case $i=1$.

\subsection{Inequalities $(T_1)$ are verified}

The inequalities $M_2 \leq d_1$ and $M_3\leq d_1$ follow from Lemma~\ref{lem:specialloopdef}, thus
\[
\max(M_2,M_3)\leq d_1.
\]
To prove the other inequality, we distinguish several cases.

\subsubsection*{Case when $M_2 \geq1$, $M_3 \geq 1$ and $C_2^{M_2}\cap C_3^{M_3} \neq \emptyset$.}

Let $v$ be a vertex in that intersection. Then one can find two faces $F$ and $F'$ containing $v$ and such that $F\subset \bar{H_2}(C^{M_2}_2)$, $F$ shares an edge with $C^{M_2}_2$, $F'\subset \bar{H_3}(C^{M_3}_3)$ and $F'$ shares an edge with $C^{M_3}_3$. By the triangle inequality and Lemma~\ref{lem:simpleloopstructure} we have
\[
d_G(F_2,F_3) \leq d_G(F_2,F) + d_G(F,F') + d_G(F', F_3) \leq (M_2-1) + 1 + (M_3 -1),
\]
thus, $d_1 \leq M_2+M_3-1$ in that case.

\subsubsection*{Case when $M_2 \geq1$, $M_3 \geq 1$ and $C_2^{M_2}\cap C_3^{M_3} = \emptyset$.}

In that case $d_G(\bar{H_2}(C_2^{M_2}),\bar{H_3}(C_3^{M_3}))\geq 1$, and since by Lemma~\ref{lem:simpleloopstructure} we have that
\[
\bigcup_{j \leq M_3-1} A_3^j \subset \bar{H_3}(C_3^{M_3}),
\]
we deduce that $d_G(\bar{H_2}(C_2^{M_2}),\bigcup_{j \leq M_3-1} A_3^j)\geq 1$. So
\begin{equation}
\label{eq:inclusion}
\bigcup_{j \leq M_3} A_3^j \subset \bar{H_2'}(C_2^{M_2})
\end{equation}
and since $M_2 \geq 1$, we have that $B_3^{M_3+1}$ is non-empty. By Lemma~\ref{lem:simpleloopstructure} there exists a simple loop $C \subset B_3^{M_3+1}$ such that $F_2 \subset \bar{H_3'}(C)$. Thus $\bar{H_3'}(C)=\bar{H_2}(C)$, and relation~\eqref{eq:inclusion} implies that $\bar{H_2}(C_2^{M_2}) \subset \bar{H_2}(C)$. Since $C$ is disjoint from all the $C_3^k$ with $1 \leq k \leq M_3$, it cannot be of type $3$ (this would contradict the fact that $M_3$ is the maximal number of disjoint simple loops of type $3$), thus $F_1 \subset \bar{H_3}(C)$. So $C$ is of type~$2$, hence has to intersect $C_2^{M_2}$, otherwise this would contradict the fact that $M_2$ is the maximal number of disjoint simple loops of type $2$. Considering the two non-disjoint simple loops $C \subset B_3^{M_3+1}$ and $C_2^{M_2} \subset B_2^{M_2}$, one concludes by selecting two appropriate faces $F$ and $F'$ as in the previous case and applying Lemma~\ref{lem:simpleloopstructure}, which yields $d_1 \leq M_2+M_3$.

\subsubsection*{Case when $M_2=0$ or $M_3=0$.}

We first show that $M_2$ and $M_3$ cannot be both zero.

\begin{lemma}
\label{lem:singlezero}
If $M_2=0$ then $M_1 \geq 1$ and $M_3 \geq 1$.
\end{lemma}

\begin{proof}
Assume that $M_2=0$. The boundary $B^1_2$ of $F_2$ is nonempty even though it contains no simple loop of type $2$. Since $d_G(F_3,F_2)\geq1$ and $d_G(F_1,F_2)\geq1$, by Lemma~\ref{lem:simpleloopstructure}, there exist two simple loops $C$ and $C'$ contained in $B^1_2$ such that $F_3 \subset \bar{H'_2}(C)$ and $F_1 \subset \bar{H'_2}(C')$. Furthermore, $C \neq C'$ otherwise $C$ would be of type $2$. By Lemma~\ref{lem:simpleloopstructure}, this implies that $H'_2(C) \cap H'_2(C') =\emptyset$, so $C$ is a simple loop of type $3$ and $C'$ is a simple loop of type $1$. Thus $M_3 \geq 1$ and $M_1 \geq 1$.
\end{proof}

In the remainder of the proof we assume that $M_2=0$. As in the proof of Lemma~\ref{lem:singlezero}, pick $C \subset B^1_2$ a simple loop of type $3$. Since $C \subset \bar{F_2} \subset \bar{H'_3} (C_3^{M_3})$, the special loop $C_3^{M_3}$ must intersect $C$ in at least a vertex $v$, otherwise $C$ would be an $(M_3+1)$-st simple loop of type $3$ which is disjoint from all the special loops $C_3^k$ with $1 \leq k \leq M_3$. Let $e$ be an edge of $C_3^{M_3}$ containing the vertex $v$ and let $F$ be the face in $\bar{H_3}(C_3^{M_3})$ containing $e$. Then $d_G(F_3,F)=M_3-1$ and $d_G(F,F_2)=1$ so $d_1=d_G(F_2,F_3)\leq M_3$.

\subsection{Inequality $(T_2)$ is verified}

By definition of $n_3$, we have that $d_3+1-M_1=M_2+1-n_3$. Hence it follows from Lemma~\ref{lem:interaction} that $H_1(C_1^{M_1}) \cap H_2(C_2^{M_2+1-n_3})= \emptyset$. Since both these hemispheres are open, we even have $\bar{H_1}(C_1^{M_1}) \cap H_2(C_2^{M_2+1-n_3})= \emptyset$, thus $C_1^{M_1}$ is disjoint from $H_2(C_2^{M_2+1-n_3})$. Similarly $C_1^{M_1}$ is disjoint from $H_3(C_3^{M_3+1-n_2})$. The nesting of the special loops implies that $C_2^{M_2-n_3} \cup C_3 ^{M_3-n_2}$ is contained in $H_1'(C_1^{M_1})$.

\subsubsection*{Case when $\min(M_2-n_3,M_3-n_2) \geq 1$ and $\max(M_2-n_3,M_3-n_2) \geq 2$.}

We reason by contradiction and assume that $n_1 > n_2+n_3 +1$. Without loss of generality assume that $M_2-n_3\geq 1$ and $M_3-n_2 \geq 2$, hence $M_2-n_3$ and $M_3-n_2-1$ are both at least $1$. By Lemma~\ref{lem:interaction}, since $M_2-n_3+M_3-n_2-1 > d_1$, we have that $C_2^{M_2-n_3}\cap C_3^{M_3-n_2-1} \neq \emptyset$. Thus $C_2^{M_2-n_3}\cap H_3(C_3^{M_3-n_2}) \neq \emptyset$ and we can draw from $C_2^{M_2-n_3}$ and $C_3^{M_3-n_2}$ a simple closed curve of type $1$ contained inside $H_1'(C_1^{M_1})$, which produces an $(M_1+1)$-th disjoint curve of type $1$, contradiction.

\subsubsection*{Case when $M_2=n_3$ or $M_3=n_2$.}

Without loss of generality assume that $M_2=n_3$. Then
\[
n_1 \leq M_2 \leq n_3 \leq n_2+n_3+1.
\]

\subsubsection*{Case when $M_2-n_3=M_3-n_2=1$.}

Then $n_1 = n_3+1+n_2+1-d_1 \leq n_2 +n_3+1$.

\section{Graphs achieving any $\sigma(G)$}
\label{sec:classofgraphs}

In this section, given $\tau=(\mu_1,\mu_2,\mu_3,\delta_1,\delta_2,\delta_3)\in(\Z_+)^3 \times \N^3$ satisfying inequalities $(T_1)$ and $(T_2)$, we construct a graph $G$ such that $\sigma(G)=\tau$. In the generic case, the graphs $G$ will be constructed by gluing together several building blocks, most of which will be Young diagrams. We point out that we resort to the Young diagram terminology to trigger the reader's imagination and avoid defining from scratch all the building blocks, but we do not expect that there should be a deeper connection to the theory of Young diagrams. Recall that the Young diagram $Y_{(\lambda_1,\ldots,\lambda_n)}$ associated with the partition $\lambda_1 \geq \lambda_2 \geq \cdots \geq \lambda_n\geq1$ is (in French notation) the diagram consisting in $n$ rows of left-aligned square boxes where the $i$-th row counted from the bottom contains $\lambda_i$ boxes.

\subsection{A class of graphs}

To any sextuple $t=(l_1,l_2,l_3,n_1,n_2,n_3)\in \Z_+^6$ such that $n_i \leq \min (l_{i+1},l_{i+2})$ for all $1 \leq i \leq 3$, we will first associate a graph $\Gamma_t$ which has the topology of the disk. We will then obtain $G_t$ by gluing two identical copies of $\Gamma_t$ along their boundaries. However, we will only glue three disjoint arcs of the boundary of one graph with three disjoint arcs of the boundary of the other graph, hence the result will be a three-holed sphere.

We first define the following building blocks:
\begin{itemize}
 \item the \emph{connector} $K$, which is a triangle with edges called (in cyclic order) $e'_1$, $e'_2$ and $e'_3$. See the top left picture of Figure~\ref{fig:buildingblocks}.
 \item for every $1 \leq i \leq 3$, the Young diagram $Y_{(l_i)}$ consisting in a single row, called a \emph{leg}. We denote its vertical left edge by $E_i$, its vertical right edge by $e_i$, its bottom (resp. top) horizontal edges from right to left by $f_{i,i+1}^1,\ldots,f_{i,i+1}^{l_i}$ (resp. $f_{i,i-1}^1,\ldots,f_{i,i-1}^{l_i}$). See the top right picture of Figure~\ref{fig:buildingblocks}.
 \item for every $1 \leq i \leq 3$, the Young diagram $Y_{(n_i,n_i-1,\ldots,2,1)}$ consisting in $n_i$ rows, called a \emph{web}. We denote its horizontal edges on the bottom boundary from left to right by ${f'}_{i+1,i+2}^1,\ldots,{f'}_{i+1,i+2}^{n_i}$ and its vertical edges on the left boundary from bottom to top by ${f'}_{i+2,i+1}^1,\ldots,{f'}_{i+2,i+1}^{n_i}$. See the bottom picture of Figure~\ref{fig:buildingblocks}.
\end{itemize}


\begin{figure}[htbp]
\centering
\includegraphics[height=4in]{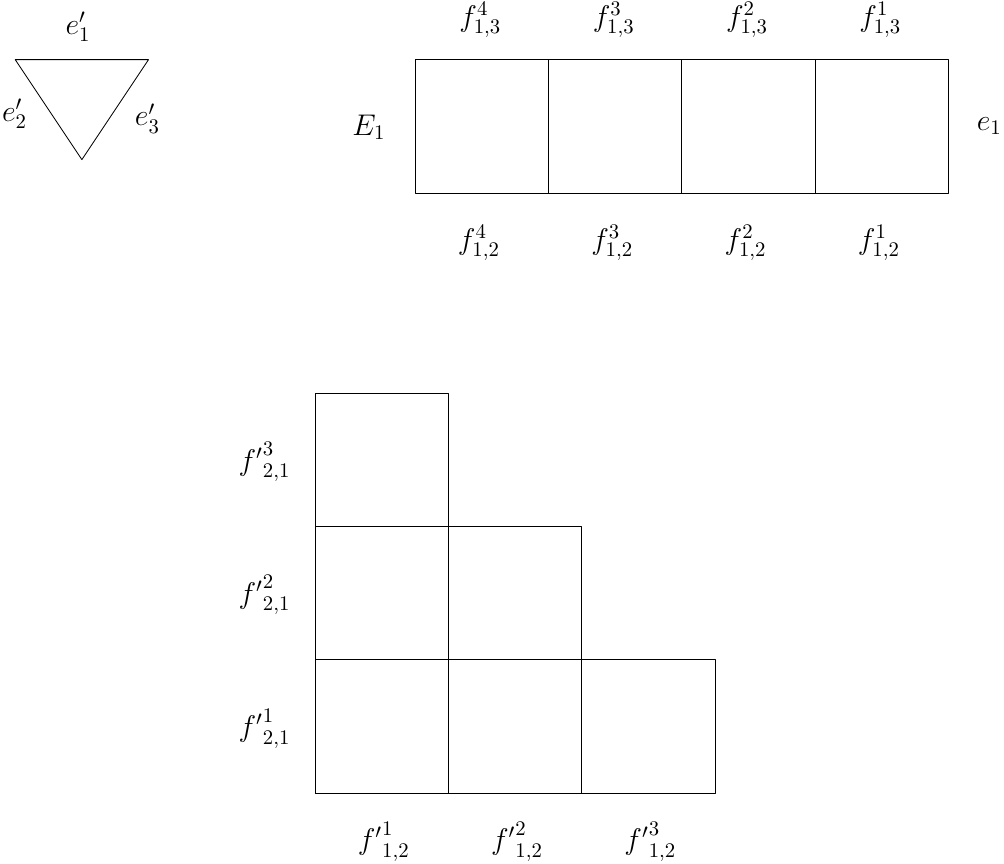}
\caption{Illustration of the different types of building blocks for the graph $\Gamma_t$: the connector $K$ (top left), the leg corresponding to $i=1$ when $l_1=4$ (top right) and the web between $1$ and $2$ when $n_3=3$ (bottom).}
\label{fig:buildingblocks}
\end{figure}

Next, for every $1 \leq i \leq 3$, we glue the edges $e_i$ with $e'_i$ and for every $1 \leq k \leq n_i$ we glue $f_{i+1,i+2}^k$ with ${f'}_{i+1,i+2}^k$ and $f_{i+2,i+1}^k$ with ${f'}_{i+2,i+1}^k$. After gluing two edges together, the result is a single edge. See Figure~\ref{fig:webbedgraph} for an example. We call the resulting graph $\Gamma_t$. It has the topology of the disk, with three distinguished edges $E_1$, $E_2$ and $E_3$ on its boundary.

\begin{figure}[htbp]
\centering
\includegraphics[height=2.5in]{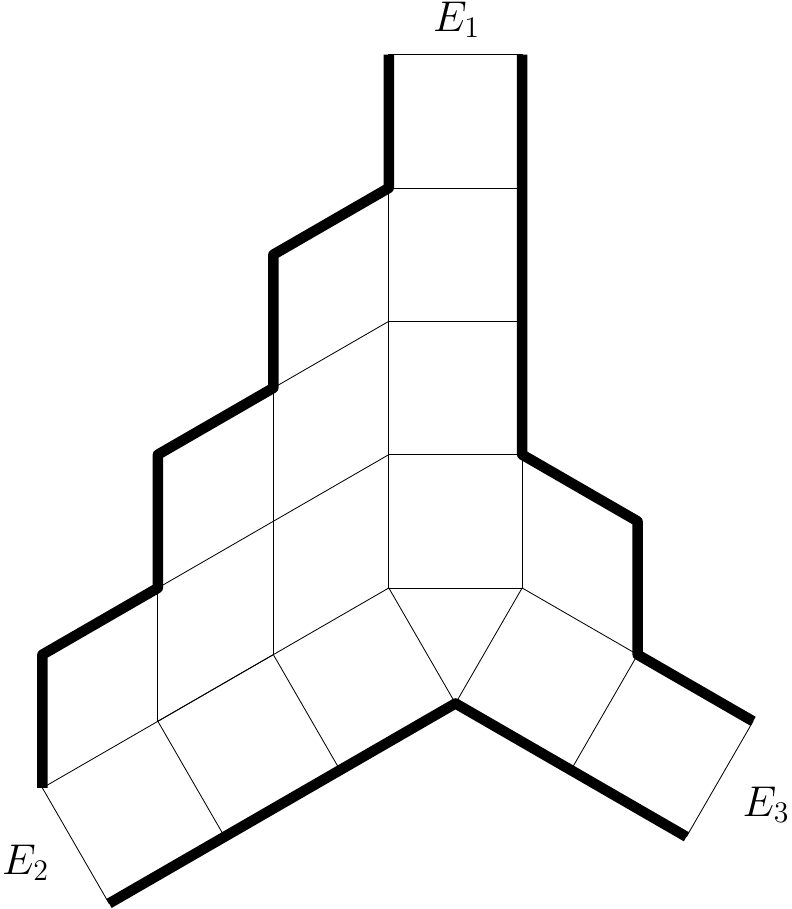}
\caption{The graph $\Gamma_t$ for $t=(4,3,2,0,1,3)$. The arcs of the boundary that will be glued to the corresponding arcs of an identical copy appear in bold stroke.}
\label{fig:webbedgraph}
\end{figure}

Let $\Gamma'_t$ be an identical copy of $\Gamma_t$. Each edge of the boundary of $\Gamma'_t$ is in canonical correspondence with an edge of the boundary of $\Gamma_t$. In particular, $\Gamma'_t$ has three distinguished edges $E'_1$, $E'_2$ and $E'_3$ on its boundary. We glue together each pair of corresponding edges, except the three pairs containing the distinguished edges. We call the resulting graph $G_t$. For every $1 \leq i\leq 3$ we denote by $F_i$ the digon with edges $E_i$ and $E'_i$.

One can compute the components of $\sigma(G_t)$ explicitly.

\begin{lemma}
\label{lem:sigmaGt}
Let $t=(l_1,l_2,l_3,n_1,n_2,n_3)\in \Z_+^6$ such that $n_i \leq \min (l_{i+1},l_{i+2})$ for all $1 \leq i \leq 3$. Then for any $1 \leq i \leq 3$,
\begin{align}
M_i(G_t) &= 1 + l_i + \max\left(0,\floor{\frac{n_i-\max(n_{i+1},n_{i+2})}{2}}\right) \\
d_i(G_t) &= 1 + l_{i+1} + l_{i+2} -n_i.
\end{align}
\end{lemma}

\begin{proof}
The proof consists in exploring $G_t$ layer by layer from a face $F_i$ and constructing explicitly the special loops defined in Section~\ref{sec:specialloops}. Looking at the type of each loop yields the desired conclusion. By symmetry of the graph $G_t$, it actually suffices to explore the graph $\Gamma_t$ layer by layer starting from an edge $E_i$ and draw the arcs corresponding to the boundary of each layer. Considering the endpoints of these arcs on the boundary of $\Gamma_t$ reveals their type when they are glued with a symmetric copy of themselves to form loops in $G_t$.
\end{proof}

\subsection{End of the proof of Theorem~\ref{thm:charac}}

Fix $\tau=(\mu_1,\mu_2,\mu_3,\delta_1,\delta_2,\delta_3)\in(\Z_+)^3 \times \N^3$ satisfying $(T_1)$ and $(T_2)$. Equivalently one can define $\nu_i:=\mu_{i+1}+\mu_{i+2}-\delta_i$, use the $\mu_i$ and $\nu_i$ as variables and require that they satisfy \eqref{eq:nureformulation}. Up to permuting the indices, one may assume that $\nu_3 \leq \nu_2 \leq \nu_1$. We will construct a $\Sigma$-graph $G$ such that $\sigma(G)=\tau$. For this we need to distinguish several cases.

\subsubsection*{Case when $\nu_2 < \nu_1 \leq \mu_1+ \nu_2$.}

Define $t=(l_1,l_2,l_3,n_1,n_2,n_3)$ by:
\begin{align}
l_1&= \nu_2-\nu_1+\mu_1 \\
l_2&= \mu_2-1 \\
l_3&= \mu_3-1 \\
n_1&= \nu_1-1 \\
n_2&= 2\nu_2 - \nu_1 \\
n_3&= \nu_2 + \nu_3 - \nu_1.
\end{align}
Then $t\in\Z^6_+$ and $n_i\leq \max(l_{i+1},l_{i+2})$ for all $1 \leq i \leq 3$. Furthermore, by Lemma~\ref{lem:sigmaGt}, $\sigma(G_t)=\tau$.

\subsubsection*{Case when $\nu_1=\nu_2$ and $\nu_3\geq1$.}

Define $t=(l_1,l_2,l_3,n_1,n_2,n_3)$ by $l_i=\mu_i-1$ and $n_i=\nu_i-1$ for all $1 \leq i \leq 3$. Then $t\in\Z^6_+$ and $n_i\leq \max(l_{i+1},l_{i+2})$ for all $1 \leq i \leq 3$. Furthermore, by Lemma~\ref{lem:sigmaGt}, $\sigma(G_t)=\tau$.

\subsubsection*{Case when $\nu_1=\nu_2$, $\nu_3=0$ and $\mu_3\geq1$.}

We construct $G$ as on Figure~\ref{fig:firstcase} by drawing $\mu_i$ nested loops around each hole $i$ such that:
\begin{itemize}
 \item the loops around hole $1$ are disjoint from the loops around hole $2$ ;
 \item the outermost $\nu_1$ (resp. $\nu_2$) loops around hole $2$ (resp. around hole $1$) intersect the outermost $\nu_1$ (resp. $\nu_2$) loops around hole $3$.
\end{itemize}
We also add line segments to make the graph $G$ connected. Then $\sigma(G)=\tau$.

\begin{figure}[htbp]
\centering
\includegraphics[height=1.5in]{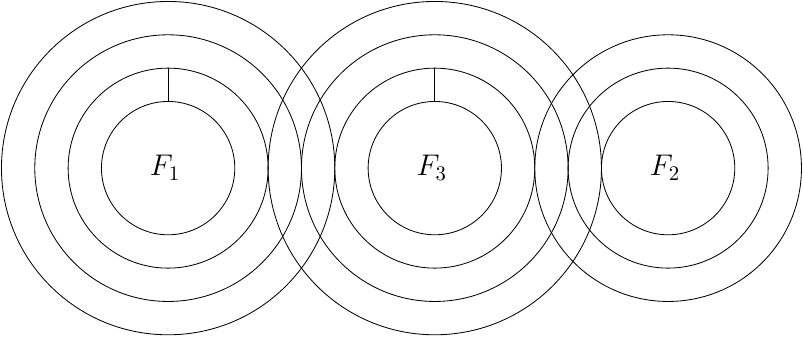}
\caption{The graph $G$ achieving $\tau=(4,3,4,4,5,7)$. Here $(\nu_1,\nu_2,\nu_3)=(3,3,0)$.}
\label{fig:firstcase}
\end{figure}

\subsubsection*{Case when $\nu_1=\nu_2$, $\nu_3=0$ and $\mu_3=0$.}

In that case, by the inequalities \eqref{eq:nureformulation}, we have that $\nu_1=\nu_2=\nu_3=\mu_3=0$. We construct $G$ by drawing $\mu_1$ nested loops around hole $1$ and $\mu_2$ nested loops around hole $2$ such that the two collections of loops are disjoint and we add a segment to each collection of nested loops to make them connected. Finally we add a single loop surrounding each collection and touching the outermost loop of each collection at a single point, see Figure~\ref{fig:secondcase}.

\begin{figure}[htbp]
\centering
\includegraphics[height=1.5in]{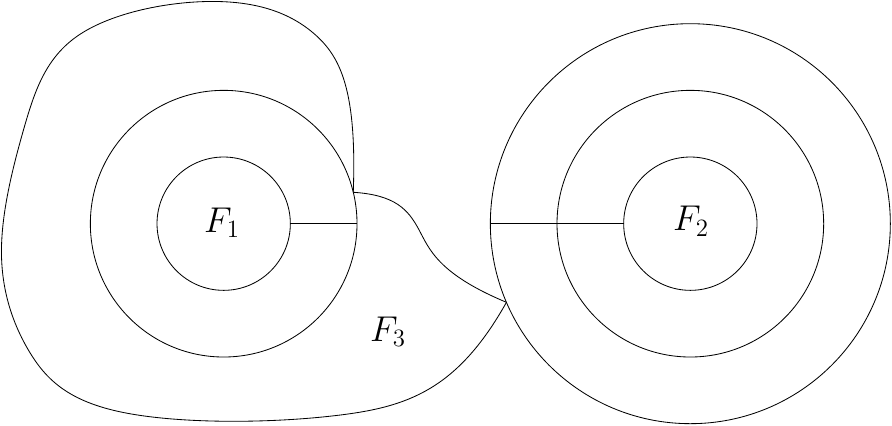}
\caption{The graph $G$ achieving $\tau=(2,3,0,3,2,5)$. Here $\nu_1=\nu_2=\nu_3=0$.}
\label{fig:secondcase}
\end{figure}

\subsubsection*{Case when $\nu_1 > \mu_1+ \nu_2$.}

It follows from \eqref{eq:nureformulation} that
\[
\nu_2+\mu_1-\nu_1 \geq \mu_1-\nu_3-1 \geq -1.
\]
Hence in the present case, we have $\nu_1=\mu_1+\nu_2+1$ and $\nu_3=\mu_1$. It also follows from \eqref{eq:nureformulation} and the fact that $\nu_2 \geq \nu_3$ that $\nu_2=\mu_1$. So $\nu_1=2\mu_1+1$ and $\nu_2=\nu_3=\mu_1$.

We construct $G$ by drawing two collections of $\mu_2$ nested loops around hole $2$ and $\mu_3$ nested loops around hole $3$ with intersection depth equal to $\nu_1$ and adding two line segments to make the graph connected. See Figure~\ref{fig:thirdcase} for an illustration. Then we have $\sigma(G)=\tau$.

\begin{figure}[htbp]
\centering
\includegraphics[height=2in]{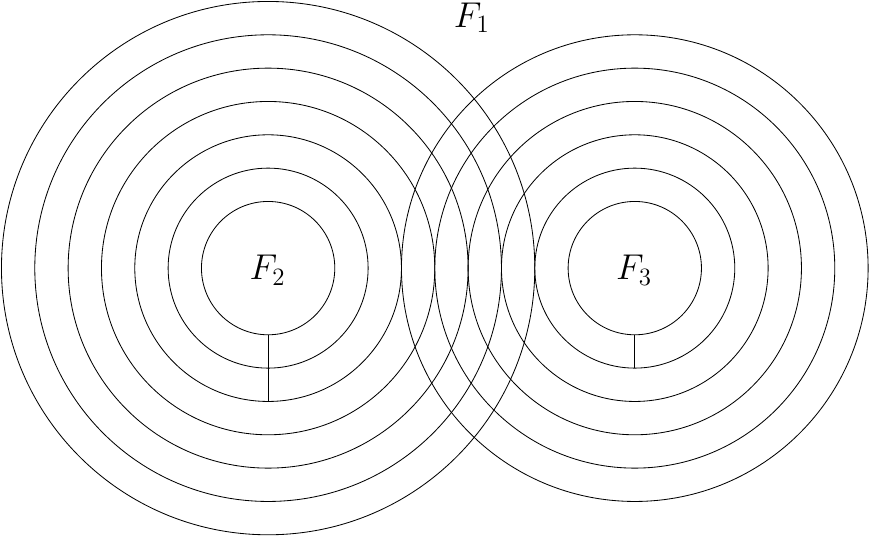}
\caption{The graph $G$ achieving $\tau=(2,7,6,8,6,7)$. Here $(\nu_1,\nu_2,\nu_3)=(5,2,2)$.}
\label{fig:thirdcase}
\end{figure}

\section*{Acknowledgements}

I thank Richard Kenyon for numerous valuable discussions throughout the course of this project, Adrien Kassel for several useful discussions and comments on an early draft of this paper and Pierre Tarrago for a fruitful conversation. I acknowledge the support of the Fondation Simone et Cino Del Duca and the Fondation Sciences Math\'ematiques de Paris during the completion of this work, as well as the hospitality of the Mathematical Sciences Research Institute in Berkeley, where this work was started during the program on ``Random spatial processes". Finally I am gratefully to the referee for advice on how to clarify the exposition.

\label{Bibliography}
\bibliographystyle{plain}
\bibliography{bibliographie}
\Addresses
\end{document}